\documentclass[11pt]{article}
\usepackage[utf8]{inputenc}
\usepackage[T1]{fontenc}

\usepackage{enumerate}
\usepackage{url}
\usepackage{epsf}
\usepackage{epsfig}
\usepackage{graphicx}
\usepackage{verbatim}

\usepackage{amsmath,amsfonts,amssymb,amsthm}
\usepackage{wasysym}
\usepackage[retainorgcmds]{IEEEtrantools}

\usepackage{stmaryrd}

\usepackage{ gensymb }
\usepackage{ esint }

\usepackage{hyperref}
\def\pathPic{Illustrations}
\def\pathCommands{Includes}
\def\pathBib{\pathCommands}

\def\cB{{\cal B}}

\def\cD{{\cal D}}

\def\cF{{\cal F}}

\def\mR{{\mathbb R}}

\def\mU{{\mathbb U}}

\def\mZ{{\mathbb Z}}

\def\R{\mR}

\def\Z{\mZ}

\def\ve{\varepsilon}
\def\vp{\varphi}

\DeclareMathOperator\Tr{Tr}

\DeclareMathOperator\diam{diam}
\DeclareMathOperator\Vor{Vor}

\def\<{\langle}
\def\>{\rangle}

\def\sm{\setminus}

\newtheorem{theorem}{Theorem}[section]
\newtheorem{remark}[theorem]{Remark}
\newtheorem{lemma}[theorem]{Lemma}
\newtheorem{proposition}[theorem]{Proposition}
\newtheorem{corollary}[theorem]{Corollary}
\newtheorem{definition}[theorem]{Definition}

\newtheorem*{proposition*}{Proposition}
\newtheorem*{theorem*}{Theorem}
\newtheorem*{lemma*}{Lemma}

\DeclareMathOperator\Hull{Hull}

\renewcommand\subset{\subseteq}
\renewcommand\supset{\supseteq}

\makeatletter
\newcommand{\oset}[2]{%
  {\mathop{#2}\limits^{\vbox to -.5\ex@{\kern-\tw@\ex@
   \hbox{\scriptsize #1}\vss}}}}
\makeatother

\usepackage{algorithm,algpseudocode,algorithmicx}

\def\RZ2{\cF(\Z^2)}

\def\WS{{\text{\tiny WS}}}

\def\FD{{\text{\tiny FD}}}

\author{
Jean-Marie Mirebeau%
\footnote{
CNRS, University Paris Dauphine, UMR 7534, Laboratory CEREMADE, Paris, France.\newline
ANR grant NS-LBR ANR-13-JS01-0003-01
}
}

\DeclareMathOperator\Leb{Leb}
\DeclareMathOperator\Newton{Newton}

\begin{document}
\title{
Discretization of the 3D Monge-Ampere operator,\\
between Wide Stencils and Power Diagrams
}
\maketitle
\date{}

\begin{abstract}
We introduce a monotone (degenerate elliptic) discretization of the Monge-Ampere operator, on domains discretized on cartesian grids. The scheme is consistent provided the solution hessian condition number is uniformly bounded. 
Our approach enjoys the simplicity of the Wide Stencil method \cite{Froese:2011ed}, but significantly improves its accuracy using ideas from discretizations of optimal transport based on power diagrams \cite{Aurenhammer:1998ie}.
We establish the global convergence of a damped Newton solver for the discrete system of equations.
Numerical experiments, in three dimensions, illustrate the scheme efficiency. 
\end{abstract}

\section{Introduction}

We introduce a discretization of the Monge-Ampere operator, on three dimensional cartesian grids, which is simultaneously monotone and consistent.
Existing consistent schemes, based e.g.\ on Finite Elements \cite{Brenner:2012ka,Neilan:2012iu} or Finite Differences \cite{Loeper:2005fn}, are not monotone, and thus require the PDE solution to be sufficiently smooth, and the numerical solver to be well initialized. 
Existing monotone schemes, based on Wide Stencil discretizations \cite{Froese:2011ed,Oberman:2006bd}, suffer from a consistency error depending on the discretization stencil angular resolution.  
Filtered schemes \cite{Froese:2013ez} combine a monotone and a consistent scheme, and attempt to cumulate their robustness and accuracy; improving either of the constituting schemes will benefit to the filtered combination.
A monotone and consistent scheme is introduced in \cite{Benamou:2014wb}, but it is limited to two dimensions. Geometric approaches \cite{Oliker:1989kz} are discussed in the third paragraph.

The proposed numerical scheme belongs to the Wide-Stencil category \cite{Oberman:2006bd}, in the sense that we actually define a family of schemes parameterized by a user chosen stencil. Larger stencils provide consistency for strongly anisotropic problems (i.e.\ for which the solution hessian is almost degenerate), but at the cost of an increased computation time.
The choice of stencil is left to the user; let us mention that in the special case of \cite{Benamou:2014wb} an automatic (solution adaptive, local, anisotropic, and parameter free) stencil construction could be designed.
Our numerical experiments show that small stencils, of radius $\sqrt 3$ or $\sqrt 6$, see the table page \pageref{table:Stencils}, already yield convincing results. The scheme is dimension independent, but we emphasize its application to three dimensional domains, which is tractable and tested.

Our approach is also inspired by \cite{Oliker:1989kz} and the discretizations of optimal transport \cite{Aurenhammer:1998ie,Merigot:2011js,Levy:2014un} based on global geometric structures, up to two differences. 
The first modification, a \emph{symmetrization} see Remark \ref{rem:Symmetrization}, is required to operate our method with the Dirichlet boundary conditions of the standard Monge-Ampere problem \eqref{eq:MAD}, instead of the second boundary conditions implicit in optimal transportation.
The second modification \emph{localizes} these methods by limiting interactions to close discretization points, see  Remark \ref{rem:Localization}, which considerably simplifies their numerical implementation.
The methods \cite{Oliker:1989kz,Aurenhammer:1998ie,Merigot:2011js,Levy:2014un} indeed rely on global geometric structures named power diagrams, which generalize Voronoï diagrams. 
Their construction requires state of the art methods of discrete geometry, which especially in 3D are still an active subject of research. For instance \cite{Levy:2014un} mentions arbitrary precision arithmetic, arithmetic filtering, expansion arithmetics and symbolic perturbation, merely for the consistent evaluation of geometric predicates. Our approach is in contrast local and requires none of these subtleties.
We show that this simplification preserves consistency in the setting of viscosity solutions, see \S \ref{subsec:Viscosity}, but at the following price: the solution hessian condition number must be uniformly bounded, and the weak Alexandroff solutions cannot be recovered.


We fix throughout this paper an open, convex and bounded domain $\Omega \subset \mR^d$, $d\geq 2$ ($d=3$ in the numerical section \S \ref{sec:Num}).  
Given a positive density $\rho \in C^0(\overline \Omega, \R_+^*)$, and some Dirichlet data $\sigma \in C^0(\partial \Omega, \R)$, we set the goal of approximating numerically the unique viscosity solution \cite{Crandall:1992kn,Gutierrez:2001wq} of
\begin{equation}
\label{eq:MAD}
\begin{cases}
\det(\nabla^2 u) = \rho & \text{on } \Omega, \\
u = \sigma & \text{on } \partial \Omega,\\
u  \text{ convex,}
\end{cases}
\end{equation}
where $\nabla^2 u$ denotes the hessian matrix of $u$.
The PDE domain $\Omega$ is discretized on a cartesian grid $X$. Up to a linear change of coordinates, encoding scaling, rotation and offset, we may assume that
\begin{equation*}
X \subset \Omega \cap \mZ^d.
\end{equation*}
The discussion of the discretization of the boundary $\partial \Omega$ is postponed to Remark \ref{rem:BoundaryDiscretization}.
\begin{definition}
\label{def:DiscreteMaps}
We denote by $\mU$ the collection of maps $u : X \cup \partial \Omega \to \mR$.
\end{definition}
\begin{definition}
A discrete operator $\cD$ associates to each $u \in \mU$ a discrete map $\cD u : X \to \mR$.
\end{definition}
Given some discretization $\cD$ of the Monge-Ampere operator, the counterpart of \eqref{eq:MAD} takes the form: find $u \in \mU$ such that 
\begin{equation}
\label{eq:DiscreteSys}
\begin{cases}
\cD u = \rho & \text{on } X,\\
u = \sigma & \text{on } \partial \Omega.
\end{cases}
\end{equation}
The constraint ``$u$ convex'' is not spelled explicitly in \eqref{eq:DiscreteSys}, contrary to \eqref{eq:MAD}, but some discrete counterpart of it often follows from the identity $\cD u = \rho$ \cite{Benamou:2014wb}.
Before discretizing the Monge-Ampere operator, we need to introduce a more basic operator $\Delta_e$, $e \in \mZ^d$, aimed at approximating the second order difference $ \<e, (\nabla^2 u (x)) e\>$. If $x \in X$, $x+e \in X$ and $x-e \in X$, then we set classically
\begin{equation}
\label{eqdef:Delta}
\Delta_e u(x) := u(x+e) - 2 u(x) + u(x-e).
\end{equation}
In general, one may not have $x+e \in X$, for instance if $x+e$ lies outside $\Omega$. Hence we introduce 
\begin{equation}
\label{eqdef:hxe}
h_x^e := \min \{ h>0; \, x+h e \in X \cup \partial \Omega\}.
\end{equation}
Let $h^+ := h_x^e$ and $h^- := h_x^{-e}$. Only one linear combination of $u(x)$, $u(x+h^+ e)$ and $u(x-h^-e)$ is consistent with $\<e, (\nabla^2 u (x)) e\>$, namely
\begin{equation}
\label{eqdef:DeltaBoundary}
\Delta_e u (x) := \frac 2 {h^++h^-} \left(\frac {u(x+h^+ e)-u(x)}{h^+}+ \frac{u(x-h^{-} e)-u(x)}{h^-}\right). 
\end{equation}

 We next illustrate the notions of consistency and monotonicity (degenerate ellipticity) using two now classical discretizations of the Monge-Ampere operator.
The Finite Differences (FD) \cite{Loeper:2005fn} approximation of the Monge-Ampere operator is defined as the determinant of a naive but consistent approximation of the hessian matrix of $u$ by finite differences: denoting by $(e_i)_{1 \leq i \leq d}$ the canonical basis of $\mR^d$
\begin{equation}
\label{eqdef:FD}
\cD^\FD u (x) := \det(\delta_{ij})_{1 \leq i, j \leq 3}, \quad \text{ with }
\delta_{ij} = 
\begin{cases}
\Delta_{e_i} u(x) & \text{ if } i=j,\\
\frac 1 4 \left(\Delta_{e_i+e_j} u(x) - \Delta_{e_i-e_j} u(x)\right) &\text{ if } i \neq j.
\end{cases}
\end{equation}
The Wide Stencil (WS) \cite{Froese:2011ed} approximation of the Monge-Ampere operator is defined as follows: denoting by $\cB \subset (\mZ^d)^d$ a finite collection of $d$-plets of pairwise orthogonal vectors
\begin{equation}
\label{eqdef:WS}
\cD^\WS_\cB u (x):= \min_{B \in \cB} \prod_{e \in B}  \frac{\max \{\Delta_e u(x), 0\} }{\|e\|^2}.
\end{equation}
\begin{definition}
An operator $\cD$ is said Degenerate Elliptic of second order (DE2), with stencil $V$, iff for all $x \in X$,  $\cD u(x)$ is a non-decreasing function of the second order differences $(\Delta_e u(x))_{e \in V}$. 
\end{definition}
By construction, scheme WS is DE2, with stencil $\{e\in \mZ^d; \, \exists B \in \cB, \, e\in B\}$. Scheme FD in contrast  is \emph{not} DE2. Degenerate Ellipticity comes with strong guarantees including a maximum principle for the discrete solutions of \eqref{eq:DiscreteSys}, and guaranteed convergence of Euler iterative solvers for this discrete system \cite{Oberman:2006bd}. It allows to recover non-smooth viscosity solutions of \eqref{eq:MAD}, see \S \ref{subsec:Viscosity}, and plays a crucial role in the proof \S\ref{subsec:Newton} of the global convergence of a damped Newton solver for \eqref{eq:DiscreteSys}.

Let $S_d$ denote the set of symmetric $d\times d$ matrices, and let $S_d^+ \subset S_d$ be the subset of positive definite matrices.
In order to analyze the consistency of these schemes, we introduce for each $M \in S_d^+$ the norm $\|\cdot\|_M$ and the map $u_M\in \mU$ defined by
\begin{align}
\label{eqdef:MNorm}
\|x\|_M &:= \sqrt{\<x, M x\>}, & u_M(x) := \frac 1 2 \|x\|_M^2.
\end{align}
\begin{definition}
\label{def:ConsistencySet}
The consistency set of an operator $\cD$ is the collection of all $M \in S_d^+$ such that 
$\cD u_M = \det(M)$, identically on $X$.
\end{definition}
The consistency set of scheme FD is the whole $S_d^+$; in fact, the identity $\cD u_M(x) = \det(M)$ also holds for non-positive matrices, although they are irrelevant for our application. The consistency set of scheme WS is in contrast of \emph{empty interior}; precisely, it consists of matrices $M \in S_d^+$ for which some $B \in \cB$ is an eigenbasis \cite{Oberman:2006bd}. Stated otherwise, the consistency of scheme WS is an asymptotic property, obtained by growing the stencil size to infinity.

A combination 
$\cD u (x) =  w(u,x) \cD^\WS_\cB u(x)+ (1-w(u,x)) \cD^\FD u(x)$ 
of schemes FD and WS is considered in \cite{Froese:2013ez}. The weight $w(u,x) \in [0,1]$ depends on the local behavior of $u$ close to $x$, as well as on the discretization scale and the angular resolution of the stencil $\cB$. Strictly speaking, the weighted scheme loses both the degenerate ellipticity of $\cD^\WS_\cB$, and the consistency of $\cD^\FD$. 
Nevertheless \cite{Froese:2013ez} establishes convergence in the setting of viscosity solutions, and illustrates numerically (in two space dimensions) that the weight $w(u,x)$ favors the consistent scheme $\cD^\FD$ except close to the most singular features of the solution. Second order accuracy may in principle be achieved even on a degenerate equation, contrary to the method proposed below. The filtered scheme construction proposed in \cite{Froese:2013ez} is quite flexible, and could be improved by replacing the monotone scheme WS with the more accurate, and still monotone, scheme proposed in this paper.

We propose an alternative solution to the apparent conflict between 
between degenerate ellipticity and consistency. The coordinates of a vector $e=(a_1, \cdots, a_d) \in \mZ^d$ are said co-prime iff $\gcd(a_1, \cdots, a_d)=1$.
\begin{definition}
\label{def:DV}
	We limit our attentions to stencils $V \subset \mZ^d \sm \{0\}$ which are finite, symmetric w.r.t the origin, span $\mR^d$, and which elements have co-prime coordinates. The proposed operator is  
\begin{equation}
\label{eqdef:LBR}
\cD_V u(x) := \Leb \{g \in \mR^d; \, \forall e \in V, \, 2\<g,e\> \leq \Delta_e u(x)\},
\end{equation}
where $\Leb$ denotes $d$-dimensional Lebesgue measure.
\end{definition}
The Degenerate Ellipticity of $\cD_V$ is clear: if any second order difference $\Delta_e u(x)$ increases, then the convex polytope appearing in \eqref{eqdef:LBR} increases for inclusion, hence also in volume. Refining slightly this argument we obtain the derivative of $\cD_V u(x)$ with respect to $\Delta_e u(x)$: namely the $(d-1)$-dimensional measure of the facets of \eqref{eqdef:LBR} defined by the equality constraint $2 |\<g,e\>| = \Delta_e u(x)$, divided by $\|e\|$ (as a result, $\cD_V u(x)$ is continuously differentiable in $u$, in contrast with \eqref{eqdef:WS}). 
Computing polytopes defined by linear inequalities like \eqref{eqdef:LBR} is, by convex duality, equivalent to computing the convex envelope of a set of points \cite{Preparata:2012uk}. Numerous computer libraries are available for that purpose, such as TetGen\textsuperscript\textregistered or CGAL\textsuperscript\textregistered (the author made his own routine which takes advantage of the symmetry of the polytope \eqref{eqdef:LBR}).

The expression \eqref{eqdef:LBR} seems more complex to evaluate than \eqref{eqdef:FD} or \eqref{eqdef:WS}. Yet in our numerical experiments \S \ref{sec:Num}, the cost of evaluating the operator $\cD_V$ was dominated by the cost of solving linear systems (when solving \eqref{eq:DiscreteSys} with a damped Newton solver), for which the MUMPS\textsuperscript\textregistered library was used.
The applicability of Newton's method to the Monge-Ampere problem is also investigated in \cite{Loeper:2005fn}, for the continuous problem, and in \cite{Froese:2011ed} and \cite{Neilan:2012iu}, for some discretizations. Its super-linear rate of convergence, in the neighborhood of the problem solution, makes it appealing for applications. In contrast, gradient descent \cite{Aurenhammer:1998ie} or first order Euler \cite{Oberman:2006bd} schemes converge slower but benefit from global convergence guarantees with monotone discretizations (iterates converge to the unique solution of the discrete problem, independently of initialization). 
For the discretization of interest, we prove \S \ref{subsec:Newton} that a damped Newton method benefits from \emph{both} a super-exponential rate of convergence close to the discrete problem solution, and a global convergence guarantee.

The following results,  established \S \ref{subsec:Consistency}, show that the consistency set of the proposed scheme is of non-empty interior, in contrast with scheme WS. By Corollary \ref{corol:FiniteStencil}, a finite stencil is sufficient to guarantee consistency for all matrices $M \in S_d^+$ with condition number below a given bound. 
\begin{definition}
\label{def:Voronoi}
For each matrix $M \in S_d^+$, we introduce the Voronoi cell and facet
\begin{align*}
\Vor(M) &:= \{g \in \mR^d;\, \forall e \in \mZ^d, \, \|g\|_M \leq \|g-e\|_M\}, \\
\Vor(M; e) &:= \{g \in \Vor(M); \, \|g\|_M = \|g-e\|_M\}.
\end{align*}
An $M$-Voronoi vector is an element $e \in \mZ^d\sm \{0\}$ such that $\Vor(M; e)\neq \emptyset$; it is said \emph{strict} iff the facet $\Vor(M;e)$ is $(d-1)$-dimensional.
\end{definition}

\begin{figure}
\centering
	\includegraphics[width=3cm]{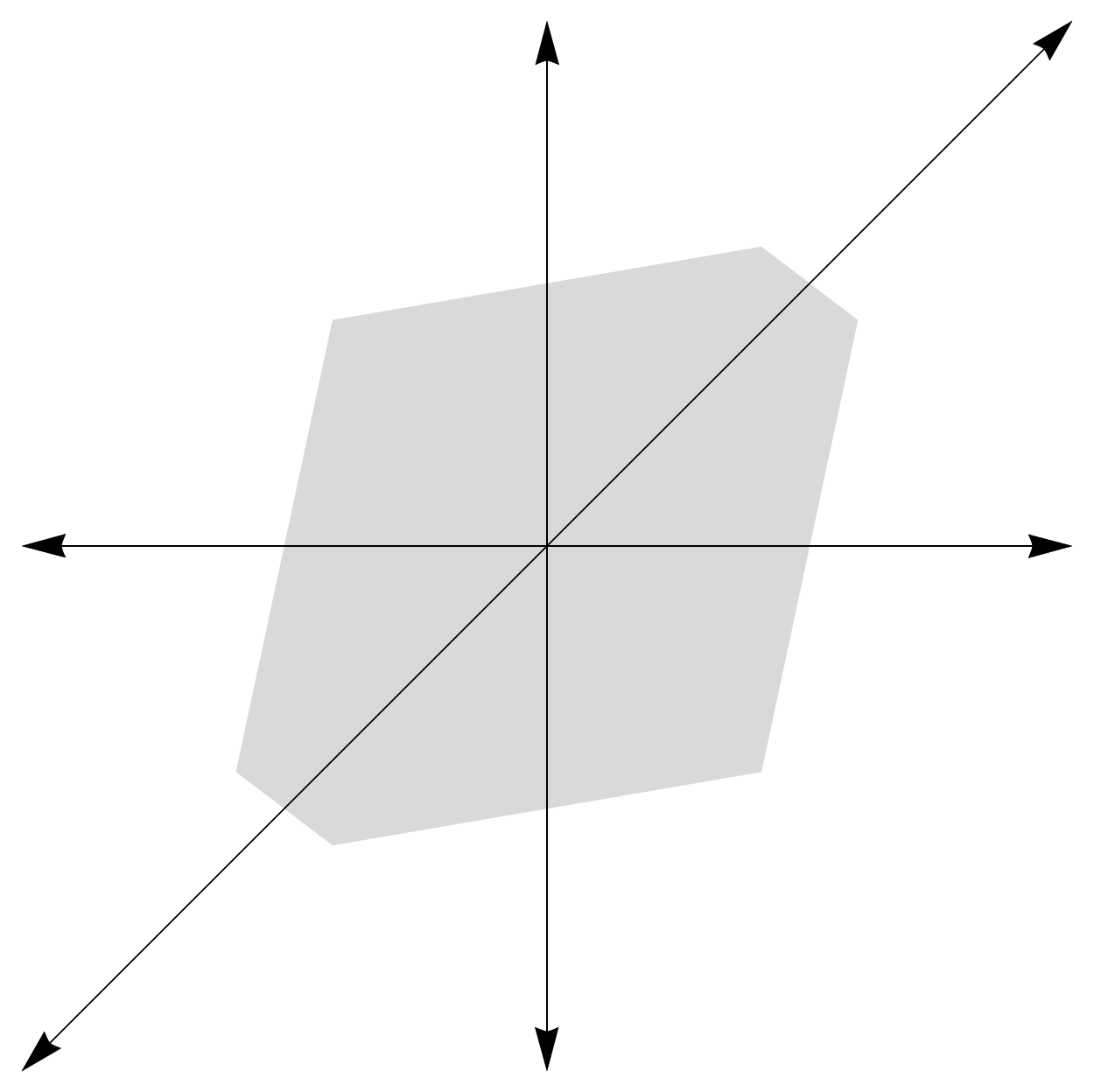}
	\includegraphics[width=3cm]{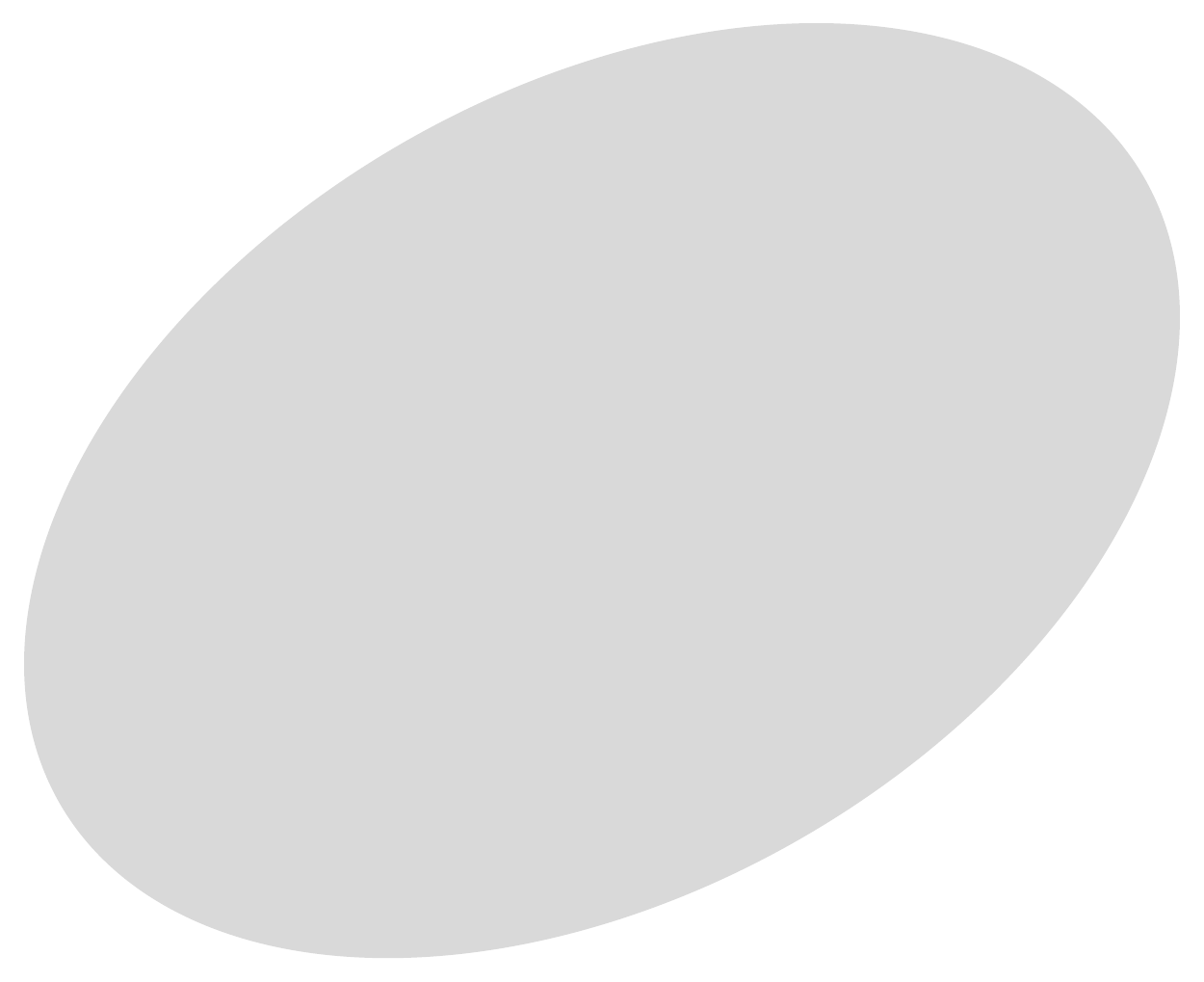}
	\includegraphics[width=5.5cm]{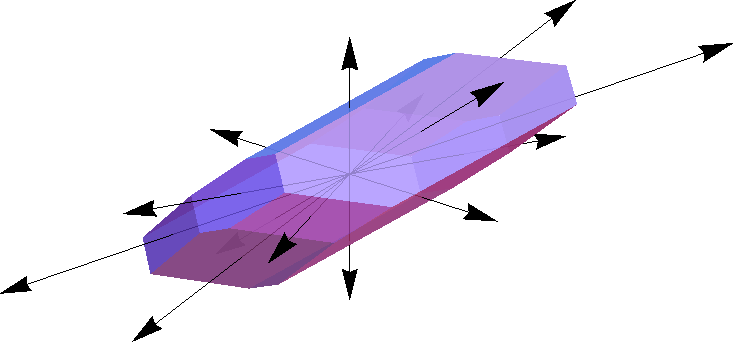}
	\includegraphics[width=4cm]{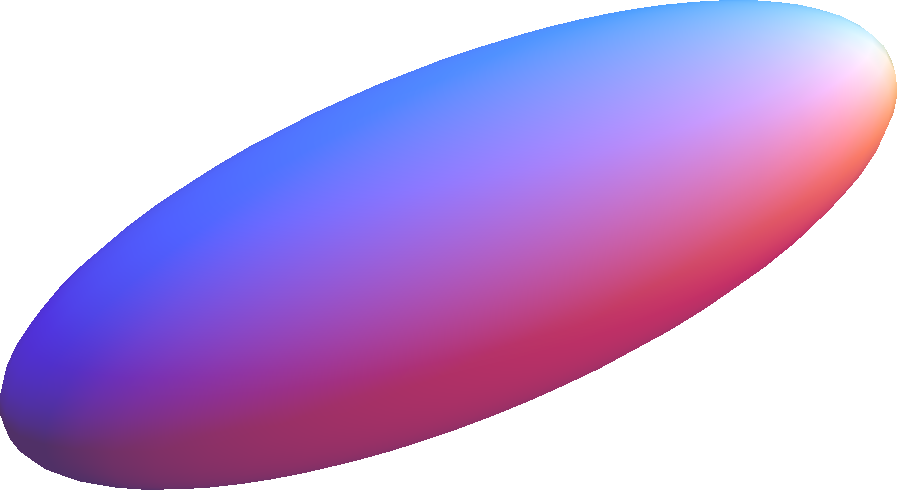}
	\caption{Voronoi cell $\Vor(M)$, Voronoi vectors $e$, and ellipse $\{x \in \mR^d; \, \<x, Mx\> \leq 1\}$. Left: 2D, Right: 3D. Note that $e$ traverses the facet $\Vor(M;e)$ at the point $e/2$.}
\end{figure}

\begin{proposition}[Consistency]
\label{prop:Consistency}
A matrix $M$ is in the consistency set of $\cD_V$ iff $V$ contains all strict $M$-Voronoi vectors. 
\end{proposition}

\begin{corollary}[Finite stencils are enough]
\label{corol:FiniteStencil}
Let $\kappa \geq 1$ and let $V$ collect all elements $e \in \mZ^d$ with norm $\|e\| \leq \kappa \sqrt d$ and co-prime coordinates. Then 
the consistency set of $\cD_V$ contains all $M \in S_d^+$ such that $\kappa \geq \sqrt {\|M\| \|M^{-1}\|}$.
\end{corollary}

\paragraph{Outline: }
The results on consistency are established \S \ref{subsec:Viscosity}, and the global convergence of the damped Newton solver in \S \ref{subsec:Newton}. We also study in \S \ref{subsec:Viscosity} convergence as the grid scale tends to zero, in the setting of viscosity solutions. Numerical experiments \S \ref{sec:Num} illustrate the method's efficiency.

\begin{remark}[Symmetrization]
\label{rem:Symmetrization}
A non-symmetrical variant of our Monge-Ampere operator discretization \eqref{eqdef:LBR} can be defined as follows: 
 for $x \in X$ such that $x+e \in X$ for all $e \in V$
\begin{equation}
\label{eqdef:AsymmetricOperator}
\cD'_V u(x) := \Leb\{ g \in \mR^d; \, \forall e \in V, \, \<g,e\> \leq u(x+e)-u(x)\}.
\end{equation}
In the case of a quadratic function $u_M$, $M \in S_d^+$, the polytope appearing in \eqref{eqdef:AsymmetricOperator} is merely a translation of \eqref{eqdef:LBR} by the vector $M x$, so that $\cD'_V u_M = \cD_V u_M$. 
In the case of a general $u$, the polytope \eqref{eqdef:AsymmetricOperator} can be non-symmetric, in contrast with \eqref{eqdef:LBR}. This asymmetry can actually help extract weak Alexandroff solutions of \eqref{eq:MAD} with non-symmetric subgradient sets, such as the two Diracs test case in \cite{Benamou:2014vl}.

The drawback of the expression \eqref{eqdef:AsymmetricOperator} is that it is only correctly defined sufficiently far from the boundary $\partial \Omega$. There is no simple way, to the knowledge of the author, to incorporate in \eqref{eqdef:AsymmetricOperator} the values of $u$ on $\partial \Omega$, and remain consistent with the Monge-Ampere operator $\det(\nabla^2 u)$ in the strong sense of Definition \ref{def:ConsistencySet}. Hence our choice of \eqref{eqdef:LBR}, defined in terms of the symmetric second order differences $\Delta_e u(x)$, $e \in V$.
\end{remark}

\begin{remark}[Localization]
\label{rem:Localization}
Consider the point dependent, largest possible stencil $V(x) = \{e \in \mZ^d; \ x+e \in X\}$. Then $\cD'_{V(x)} u(x)$, see  \eqref{eqdef:AsymmetricOperator}, is the measure of the power cell associated to $x$ in the power diagram based discretization \cite{Oliker:1989kz,Aurenhammer:1998ie,Merigot:2011js,Levy:2014un} of optimal transport. Since the stencils $(V(x))_{x \in X}$ are typically huge, it is not reasonable to evaluate $\cD'_{V(x)} u(x)$ independently for each $x$. Instead a global power diagram needs to be constructed, using complex geometric procedures which validity requires extremely careful evaluations of geometric predicates \cite{Levy:2014un}.
\end{remark}


\section{Proofs of the main results} 
\label{sec:Proofs}

We establish in \S \ref{subsec:Consistency} the numerical scheme consistency, show in \S \ref{subsec:Newton} the global convergence of a damped Newton solver for the discrete system, and prove in \S \ref{subsec:Viscosity} a convergence result in the setting of viscosity solutions as the discretization grid scale tends to zero.

\subsection{Consistency}
\label{subsec:Consistency}

Our consistency analysis relies on elementary properties of Voronoi cells of lattices, see Definition \ref{def:Voronoi}, and \cite{Conway:1992cq} for more on this topic. Our first step is to bound their diameter.
Let $\kappa(M)^2$ denote the condition number of a matrix $M \in S_d^+$, i.e.\ $\kappa(M) := \sqrt{\|M\| \|M^{-1}\|}$.
\begin{lemma}
\label{lem:VoronoiRadius}
Let  $M \in S_d^+$. Any point $g \in \Vor(M)$ satisfies $\|g\| \leq \frac 1 2 \kappa(M) \sqrt d$. Any $M$-Voronoi vector $e$ satisfies $\|e\| \leq \kappa(M) \sqrt d$, and has co-prime coordinates.
\end{lemma}
\begin{proof}
\emph{Bound on the norm.} Let $g \in \Vor(M)$, and let $e_g$ be obtained by rounding the coordinates of $g$ to the nearest integer, so that $\|g-e_g\| \leq \frac 1 2 \sqrt d$. One has $\lambda^- \|g\| \leq \|g\|_M \leq \|g-e_g\|_M \leq \lambda^+ \|g-e_g\|$, where $\lambda^-, \lambda^+$ are respectively the smallest and largest eigenvalues of $M^\frac 1 2$. Thus $\|g\| \leq  (\lambda^+/\lambda^-) \frac 1 2 \sqrt d$ as announced. 
If $e$ is an $M$-Voronoi vector, then there exists $g \in \Vor(M)$ such that $\|g\| = \|e-g\|$, hence $\|e\| \leq 2 \|g\| \leq  \kappa(M) \sqrt d$.

\emph{Coordinates are co-prime.} Consider a vector which coordinates are not co-prime: $k e$, with $k \geq 2$ and $e \in \mZ^d\sm \{0\}$. Then for any $g \in \mR^d$ one has $\|k e-g\|_M^2 + (k-1) \|g\|_M^2 = k \|e-g\|_M^2+(k^2-k)\|e\|_M^2$. Hence $\|e-g\|_M < \max \{\|k e-g\|_M, \|g\|_M\}$ and therefore $k e$ cannot be an $M$-Voronoi vector.
\end{proof}

\begin{corollary}
\label{corol:InclusionEquality}
Let $M \in S_d^+$, and let $E$ be the collection of strict $M$-Voronoi vectors. 
For any set $V \subset \mZ^d$ one has $\Vor(M) \subset \{g \in \mR^d; \, \forall e \in V, \, 2\<g, M e\> \leq \|e\|_M^2\}$, with equality iff $E \subset V$.
\end{corollary}

\begin{proof}
For any $g , e \in \mR^d$, we have the equivalence $\|g\|_M \leq \|g-e\|_M \Leftrightarrow 2\<g,M e\> \leq \|e\|_M^2$, obtained by squaring both sides of the first inequality.
Hence $\Vor(M)$ is a convex polytope, defined by a family of linear inequalities indexed by $e \in \mZ^d$. By Lemma \ref{lem:VoronoiRadius} one can eliminate all inequalities but a finite number. Among these inequalities, only those corresponding to strict $M$-Voronoi vectors are active, in the sense that they define a facet of $\Vor(M)$. The result follows.
\end{proof}

We next identify the volume of a Voronoi cell, and use it to establish our scheme consistency.

\begin{lemma}
\label{lem:VorVol}
For any $M\in S_d^+$ one has $\Leb(\Vor(M))=1$.
\end{lemma}
\begin{proof}
The set $\Vor(M)$ collects elements $g \in \mR^d$ which are closer to the origin than to any other point $e \in \mZ^d$. As a result the translates of $\Vor(M)$, by all offsets $e \in \mZ^d$, tile the space $\mR^d$ up to a negligible set. Thus  $\Leb(\Vor(M))$ is the co-volume of the lattice $\mZ^d$, namely $1$ as announced.
\end{proof}

\begin{proof}[Proof of Proposition \ref{prop:Consistency} (Consistency)]
Let $M \in S_d^+$, let $x \in X$, and let $u_M$ be the quadratic map \eqref{eqdef:MNorm}.
By construction $\Delta_e u_M (x)= \|e\|_M^2$, for any $e \in \mZ^d$. Hence $\cD_V u_M(x)$ is the volume of: \begin{align*}
\{g \in \mR^d;\, \forall e \in V, \, 2\<g,e\> \leq \|e\|_M^2\} = M \{g \in \mR^d;\, \forall e \in V, \, 2\<g, M e\> \leq \|e\|_M^2\} \supset M \Vor(M).
\end{align*}
where we used Corollary \ref{corol:InclusionEquality} and abusively denoted by $M E := \{ M e ; \, e \in E\}$ the image of a set $E$ by the linear map $M$.
Therefore $\cD_V u_M(x) \geq \det(M) \Leb(\Vor(M)) = \det(M)$ by Lemma \ref{lem:VorVol}. Equality holds iff the above inclusion of polytopes is an equality, equivalently iff $V$ contains all strict $M$-Voronoi vectors by Corollary \ref{corol:InclusionEquality}.
\end{proof}

\begin{proof}[Proof of Corollary \ref{corol:FiniteStencil} (Finite stencils are enough)]
Combine Proposition \ref{prop:Consistency} and Lemma \ref{lem:VoronoiRadius}.
\end{proof}

\subsection{Global convergence of a damped Newton solver}
\label{subsec:Newton}

We establish the convergence of a damped Newton solver for the discrete system \eqref{eq:DiscreteSys}. Before doing so, we need to clarify the implementation of boundary conditions. 
The stencil $V$ is fixed in the following, and obeys the properties of Definition \ref{def:DV}.
\begin{remark}[Boundary discretization]
\label{rem:BoundaryDiscretization}
Maps $u \in \mU$ are in principle, see Definition \ref{def:DiscreteMaps}, defined both on the discrete sampling $X$ of $\Omega$, and on the whole uncountable boundary $\partial \Omega$. 
Fortunately, only finitely many values of $u$ on $\partial \Omega$ play an active role in the system of equations \eqref{eq:DiscreteSys}: those on 
\begin{equation*}
\partial X := \{ x+h_x^e e; \, x \in X, \, e \in V\} \cap \partial \Omega.
\end{equation*}
This discretization, with $\#(X \cup \partial X)$ unkowns, 
is used in our experiments, and in the proof of the damped Newton solver global convergence below.

Alternatively one may (i) not introduce any unknown for the boundary values, but use the given Dirichlet data $\sigma$ as in \cite{Benamou:2014wb}, or (ii) introduce some unknowns on an arbitrary boundary sampling $\partial X'$, extended to $\partial \Omega \sm \partial X'$ by some interpolation procedure. Unfortunately (i) leads to initialization difficulties for the iterative solver, since one needs to find a strictly convex seed $u_0$ with prescribed boundary values on $\partial \Omega$, and (ii) may limit the accuracy of the scheme if first order interpolation is used, or violate its degenerate ellipticity in the case of high order interpolation.
\end{remark}

Our first step is to establish, in Corollary \ref{corol:Invertibility}, that the Jacobian matrix associated to the discrete Monge-Ampere system \eqref{eq:DiscreteSys} is invertible. This follows from the invertibility of diagonally dominant matrices, recalled in the next lemma, and from the degenerate ellipticity of the proposed operator $\cD_V$.

\begin{lemma}
\label{lem:PerronFrobenius}
Let $A$ be an $n \times n$ matrix
such that for each $1 \leq i \leq n$ 
\begin{equation}
\label{eq:DiagonallyDominant}
|A_{ii}| \geq \sum_{j \neq i} |A_{ij}|.
\end{equation}
Assume that for each index $1 \leq i_0 \leq n$ there exists a chain $i_0, \cdots, i_k$ such that $i_k$ satisfies \emph{strictly} inequality \eqref{eq:DiagonallyDominant}, and $A_{i_r, i_{r+1}} \neq 0$ for all $0 \leq r < k$. Then $A$ is invertible. 
\end{lemma}
\begin{proof}
For each index $1 \leq i \leq n$, let $k(i)$ denote the length of the smallest chain of indices as above.
Let $x \in \mR^d$ be such that $A x = 0$. Among the indices $1 \leq i \leq n$ such that the vector component $|x_i|$ is maximal, choose one which minimizes $k(i)$.
From $(Ax)_i = 0$ we obtain 
\begin{equation*}
|A_{ii}| |x_i| \leq \sum_{j \neq i} |A_{ij}| |x_j|,
\quad \text{ thus }
(|A_{ii}|-\sum_{j \neq i} |A_{ij}|) |x_i| + \sum_{j\neq i} |A_{ij}| (|x_i| - |x_j|)  = U+V \leq 0.
\end{equation*}
Both terms $U$ and $V$ are non-negative by construction, hence $U=V=0$. If $k(i)=0$, then from $U=0$ we obtain $|x_i|=0$, hence $x=0$. If $k(i)\neq 0$, then from $V=0$ we obtain $|x_j| = |x_i|$ for all indices $1 \leq j \leq n$ such that $A_{i j} \neq 0$. One of these indices satisfies $k(j) = k(i)-1$, which contradicts our choice of $i$. This concludes the proof.
\end{proof}

Let $\mU_0$ be the collection of maps $u : X \cup \partial X \to \mR$ such that: 
\begin{equation}
\label{eqdef:U0}
\forall x \in X,\ \cD_V u(x) > 0. \qquad \text{ (Equivalently: } \forall x\in X, \ \forall e \in V, \ \Delta_e u(x) > 0.)
\end{equation}
Let $f : \mU_0 \to \mR^{X \cup \partial X}$ be defined by 
\begin{equation}
\label{eqdef:fux}
f(u)(x) := 
\begin{cases}
\ln \cD_V u(x)& \text{ if } x \in X,\\ 
u(x)& \text{ if } x \in \partial X.
\end{cases}
\end{equation}

\begin{corollary}
\label{corol:Invertibility}
For each $u \in \mU_0$, the Jacobian matrix $df(u)$ is invertible.
\end{corollary}

\begin{proof}
Let $u \in \mU_0$ and let $A := df(u)$, which formally is a matrix associating a coefficient to each index pair $(x,y)$, $x,y\in X \cup \partial X$. The line of $A$ corresponding to each $x \in \partial X$ has a single non-zero coefficient, namely $1$  at index $(x,x)$, hence this line satisfies \eqref{eq:DiagonallyDominant} strictly. By degenerate ellipticity of $\cD_V$, the line corresponding to any $x \in X$ satisfies \eqref{eq:DiagonallyDominant}. 
We prove in the following the chain property of Lemma \ref{lem:PerronFrobenius}, which implies the announced invertibility.

Let $x \in X$, and let $K_x$ be the polytope appearing in \eqref{eqdef:LBR}, which by construction is convex, compact and symmetric w.r.t the origin. 
For each $e \in V$ let $F_e$ denote the facet of $K_x$ defined by the equality constraint $2 \<g,e\> = \Delta_e u(x)$. Let $V_x := \{e \in V; \, \Lambda_e>0\}$ where $\Lambda_e$ denotes the $(d-1)$-dimensional measure of $F_e$. Since the polytope is symmetric, one has $\Lambda_e = \Lambda_{-e}$ for all $e \in V$, hence $V_x$ is symmetric. Since the polytope is compact, and since the exterior normal to $F_e$ is $e/\|e\|$, the set $V_x$ spans $\mR^d$. The coefficient of the Jacobian matrix $A$ at index $(x,x+ h_x^e e)$ is $4 \Lambda_e /(h_x^e (h_x^e+h_x^{-e}) \|e\| \Leb(K_x) )$, using \eqref{eqdef:DeltaBoundary} and the geometric argument that the polytope volume variation is at first order given by the facets areas $\Lambda_e + \Lambda_{-e}$ times their normal displacement. This coefficient is hence positive if $e \in V_x$.

Among the points $x+h_x^e e$, $e \in V_x$ one at least is strictly closer to $\partial \Omega$ than $x$. By induction we can thus build a finite chain $x=x_0, \cdots, x_{k-1} \in X$ such that $x_k \in \partial X$ and the coefficient of $A$ at index $(x_r,x_{r+1})$ is non-zero for each $0 \leq r<k$. The announced result then follows from Lemma \ref{lem:PerronFrobenius}.
\end{proof}

Our second step is to show that \eqref{eqdef:fux} is a proper map: the preimage of any compact set is a compact set, a property which is tightly linked with the well-posedness of the PDE \eqref{eq:MAD}. 
Let 
\begin{equation}
\label{eqdef:Subgradient}
\partial_x U := \{g \in \mR^d; \, \forall y \in \overline \Omega, \, \<g,y-x\> \leq U(y)-U(x)\}
\end{equation}
denote the subgradient of a convex map $U\in C^0(\overline \Omega, \mR)$ at a point $x \in \Omega$. 
We connect in Lemma \ref{lem:SubgradientInclusion} the proposed operator $\cD_V u(x)$ with the Lebesgue measure of  $\partial_x U$, where $U$ is the lower convex envelope of $u$. The properness of \eqref{eqdef:fux} then follows in Corollary \ref{corol:Proper} from the maximum principle of Alexandroff-Bakelman-Pucci.

\begin{lemma}
\label{lem:SubgradientInclusion}
Let $u \in \mU_0$, and let $U : \Hull(X \cup \partial X) \to \mR$ be the maximal convex map bounded above by $u$. Then $\Leb( \partial_x U ) \leq h_*^d \cD_V u(x)$ for all $x \in X$, with 
$h_* := \max\{h_x^e; \, x\in X, e \in V\}$. 
\end{lemma}

\begin{proof}
By construction $U \leq u$, and for any $x \in X$ such that $U(x)<u(x)$ one has $\Leb( \partial_x U ) = 0$, so that the announced inequality holds.
Assume that $U(x) = u(x)$, and let $g^+, g^- \in \partial_x U$. Let $e \in V$ and let $h^+ := h_x^e$, $h^- := h_x^{-e}$ as in \eqref{eqdef:DeltaBoundary}. One has
$\<g^+, h^+ e\> \leq u(x+h^+ e)-u(x)$, and $\<g^-, - h^- e\> \leq u(x-h^- e) - u(x)$ by \eqref{eqdef:Subgradient}, thus 
\begin{equation*}
\<g^+ - g^-, e\> \leq \frac {u(x+h^+ e) - u(x)}{h^+} + \frac {u(x-h^- e)-u(x)}{h^-} = \frac {h^++h^-} 2 \Delta_e u(x) \leq h_* \Delta_e u(x).
\end{equation*}
Therefore $\frac 1 2 (\partial_x U - \partial_x U) \subset h_* K_x$, where the left hand side is a Minkowski sum of sets, and $K_x$ is the polytope appearing in \eqref{eqdef:LBR}. The announced estimate then follows from this inclusion and Brunn-Minkowski's inequality.
\end{proof}

\begin{theorem}[Alexandroff-Bakelman-Pucci's maximum principle \cite{Gutierrez:2001wq}]
\label{th:MinMin}
Let $u \in C^0(\overline \Omega)$, and let $U$ be the maximal convex map bounded above by $u$. Then 
\begin{equation}
\label{eq:MinMin}
\min_{\partial \Omega} u -\min_{\overline \Omega} u \leq  (\Leb(G)/\omega_d)^\frac 1 d \diam(\Omega), 
\qquad \text{ with }
G:= \bigcup_{\substack{x \in \Omega\\ u(x) = U(x)}} \partial_x U,
\end{equation}
where $\diam(\Omega) := \max \{\|x-y\|; x, y \in \Omega\}$,
and $\omega_d$ is the volume of the $d$-dimensional unit ball.
\end{theorem}


\begin{corollary}
\label{corol:Proper}
The map $f : \mU_0 \to \mR^{X \cup \partial X}$ is proper.  
\end{corollary}

\begin{proof}
Since $f$ is continuous, it suffices to bound by above and below the values of an arbitrary $u \in \mU_0$ in terms of those of $f(u)$. We have the obvious estimate $f(u)(x)=u(x)$ for values at $x \in \partial X$.

\emph{Lower bound on $X$.} 
Let $\Omega'$ be the interior of $\Hull(X \cup \partial X)$, and let $u'\in C^0( \overline \Omega', \mR)$ be defined by $u'(z) = \min_{x \in X\cup \partial X} u(x)+ \lambda \|x-z\|$, where $\lambda>0$. Let $U \in C^0(\overline \Omega', \mR)$ be the maximal convex function bounded above by $u'$. If the constant $\lambda$ is sufficiently large, then $U$ also is the maximal convex function bounded above by $u$. In addition for all $x \in \Omega'$, one has $U(x) = u'(x) \Leftrightarrow x \in X\cup \partial X$. 
We obtain applying Theorem \ref{th:MinMin} the desired lower bound on $\min\{ u(x); \, x\in X\}$, since
\begin{align*}
\min_{\overline {\Omega'}} u' &= \min_{X \cup \partial X} u, & 
\min_{\partial \Omega'} u' &= \min_{\partial X} u, & 
\Leb(G) & = \sum_{x \in X} \Leb(\partial_x U) \leq h_*^d \sum_{x \in X} \cD_V u(x) = h_*^d \sum_{x \in X} e^{f(u)(x)}.
\end{align*}

\emph{Upper bound on $X$.} Let $x \in X \cup \partial X$ be such that $u(x)$ is maximal, and let us assume for contradiction that $x \notin \partial X$. 
Since $\Delta_e u(x) > 0$, for any $e \in V$, one has either $u(x+h^+ e) > u(x)$, or $u(x-h^- e) > u(x)$, with the notations of \eqref{eqdef:DeltaBoundary}. This contradicts the maximality of $u(x)$, and concludes the proof.
\end{proof}

Numerous variants of the Newton method exist \cite{Eisenstat:1994wk}; an elementary one is presented below for completeness, which guarantees global convergence for the system \eqref{eq:DiscreteSys} of interest.
This damped Newton algorithm is an iterative equation solver, which recursion rule \eqref{eqdef:NewtonUpdate} involves an adaptively chosen ``damping'' parameter $\delta$. In practice $\delta$ is typically small in the first iterations, and equal to $1$ in the last ones, which coincide with the classical Newton method and enjoy quadratic convergence. Convergence is guaranteed for maps which, like \eqref{eqdef:DeltaBoundary}, are shown to be proper and at each point a local diffeomorphism.
Note that these assumptions, plus the connectedness of the source domain and the simple connectedness of the target domain (here both satisfied), imply by Hadamard-Levy's theorem that $f$ is a global diffeomorphism.



\begin{proposition}[Global convergence of the damped Newton algorithm]
\label{prop:NewtonConvergence}
Let $N>0$, let $\mU_0 \subset \mR^N$ be an open set, and let $f \in C^1(\mU_0, \mR^N)$. Assume that $f$ is proper and that the Jacobian matrix $F(x) := df(x)$ is invertible for each $x \in \mU_0$.
Consider $y \in \mR^N$, and for each $x \in \mU_0$, $\delta \in [0,1]$, the Newton update
\begin{equation}
\label{eqdef:NewtonUpdate}
\Newton(x,\delta) := x+\delta F(x)^{-1}(y-f(x)).
\end{equation}
Let $x_0 \in \mU_0$, and for each $n \geq 0$ let $x_{n+1} := \Newton(x_n,\delta_n)$, where $\delta_n = 2^{-k_n}$ and $k_n$ is the smallest non-negative integer such that $\|y-f(x_{n+1})\| \leq (1-\delta_n/2)\|y-f(x_n)\|$.
Then $f(x_n) \to y$ as $n\to \infty$.
\end{proposition}

\begin{proof}
Introduce the set $K := \{x \in \mU_0; \, \|y - f(x)\| \leq \|y-f(x_0)\|\}$, which is compact since $f$ is proper. For any $x \in \mU_0$ one has the Taylor expansion: for small $\delta$
\begin{align}
\nonumber
f(\Newton(x,\delta)) &= f(x) + \delta F(x) F(x)^{-1} (y-f(x)) + o(\delta \|F(x)^{-1}(y-f(x))\|)\\
\label{eq:NewtonNorm}
&= y - (1-\delta) (y-f(x)) + o(\delta \|y-f(x)\|).
\end{align}
Hence 
$\|y - f(\Newton(x,\delta))\| = (1-\delta + o(\delta)) \|y-f(x)\|$ using \eqref{eq:NewtonNorm}, which is smaller than $(1-\delta/2) \|y-f(x)\|$ for an open range of $\delta \in ]0, \lambda(x)[$. By compactness, this property holds  for all $x \in K$ with an uniform open range $]0, \lambda[$.

Thus $\delta_0$ is well defined, bounded below by $\lambda/2$, and by construction $\|y-f(x_1)\| \leq (1-\lambda/2) \|y-f(x_0)\|$ so that $x_1 \in K$. The result follows from an immediate induction argument.
\end{proof}

\subsection{Convergence in the setting of viscosity solutions}
\label{subsec:Viscosity}

In this section, we let the discretization grid scale tend to zero, and study the convergence of the minimizers to the discrete problems \eqref{eq:DiscreteSys}.
For all integers $n \geq 1$, let $X_n := \Omega \cap n^{-1} \mZ^d$.
Let $\mU_n$ be the collection of (semi-)discrete maps $u : X_n \cup \partial \Omega \to \mR$. 
The stencil $V$ is fixed and obeys the assumptions of Definition \ref{def:DV}.
For any $u \in \mU_n$, $x \in X_n$, and $M \in S_d$ let
\begin{align}
\nonumber
	\cD_n(u) &:= \Leb\{g \in \mR^d; \, \forall e \in V, \ 2\<g,e\> \leq \Delta_{e/n} u(x)\},\\
	\label{eqdef:DM}
	D(M) &:= \Leb\{g \in \mR^d; \, \forall e \in V, \ 2\<g,e\> \leq  \<e,Me\>\}. 
\end{align}
We denoted by $\Delta_{e/n} u(x) := n^2 (u(x+e/n) - 2 u(x)+u(x-e/n))$ the standard approximation of $\<e, (\nabla^2 u(x)) e\>$ on the grid $X_n$ (This expression assumes that  $x\pm e/n \in X_n$, and should be modified as in \eqref{eqdef:DeltaBoundary} otherwise).
Given a density $\rho \in C^0(\overline \Omega, \mR_+^*)$, and some Dirichlet data $\sigma \in C^0(\partial \Omega, \mR)$, we study the problems
\begin{align}
\label{eqdef:Pbn}
	\begin{cases}
		u \in \mU_n,\\
		\cD_n(u) = \rho & \text{ on } \Omega,\\
		u = \sigma & \text{ on } \partial \Omega.
	\end{cases}
& &
	\begin{cases}
		u : \overline \Omega \to \mR, \\ 
		D(\nabla^2 u) = \rho & \text{ on }\Omega,\\
		u = \sigma & \text{ on } \partial \Omega.
	\end{cases}
\end{align}
By \S \ref{subsec:Newton}, the discretized problem (\ref{eqdef:Pbn}, left) has a unique solution $u_n \in \mU_n$. We show in Corollary \ref{corol:ExistenceContinuous} that (\ref{eqdef:Pbn}, right) also admits a unique solution $u_\infty$.  Importantly, if $\nabla^2 u_\infty$ exists at each $x \in \Omega$ and belongs to the consistency set of $\cD_V$, then $u_\infty$ is also the unique solution to the Monge-Ampere equation \eqref{eq:MAD}. 
The study of (\ref{eqdef:Pbn}, right) relies on the concept of viscosity solutions \cite{Crandall:1992kn}.

\begin{definition}
\label{def:SubSup}
	A sub-solution (resp.\ super-solution) of (\ref{eqdef:Pbn}, right) is an upper-semi-continuous (resp.\ lower-semi-continuous) map $u : \overline \Omega \to \mR$ such that (I) $u \leq \sigma$ (resp.\ $u\geq \sigma$) on $\partial \Omega$ and (II)
	 for any $x \in \Omega$ and any $\vp \in C^2(\Omega)$ such that $u-\vp$ has a local maximum (resp.\ minimum) at $x$, one has $D(\nabla^2 \vp(x)) \geq \rho(x)$ (resp.\ $\leq \rho(x)$).
	 
	 A solution to (\ref{eqdef:Pbn}, right) is a map $u : \overline \Omega \to \mR$ which is both a super-solution and a sub-solution.
\end{definition}

Replacing ``local maximum'' with ``strict global maximum'' (resp. minimum) in Definition \ref{def:SubSup} yields an equivalent notion of sub- and super-solutions \cite{Crandall:1992kn}.
Super- and sub-solutions of (\ref{eqdef:Pbn}, right) obey a comparison principle, proved Proposition \ref{prop:ComparisonPrinciple}. 
Given symmetric matrices $M, M' \in S_d$, we write $M \preceq M'$ iff the difference $M'-M$ is positive semi-definite.

\begin{lemma}
\label{lem:MinkowskiD}
For any $M,M' \in S_d$, one has: $M \preceq M' \Rightarrow D(M) \leq D(M')$. More quantitatively, if $M, H \in S_d$ are such that $D(M)>0$ and $D(H)>0$ then 
	\begin{equation*}
		D(M+H)^\frac 1 d  \geq D(M)^\frac 1 d + D(H)^\frac 1 d.
	\end{equation*}
\end{lemma}

\begin{proof}
	Let $K(M) \subset \mR^d$ be the polytope appearing in \eqref{eqdef:DM}, so that $D(M) = \Leb K(M)$. Then
	\begin{equation*}
		M \preceq M' \ \Rightarrow \ 
		\forall e \in V, \|e\|^2_M \preceq \|e\|^2_{M'} \ \Rightarrow \
		K(M) \subset K(M') \ \Rightarrow \ D(M) \leq D(M').
	\end{equation*}
	Second point: since $\<e,(M+H) e\> = \<e,M e\>+\<e,H e\>$, 
	the set $K(M+H)$ contains the Minkowski sum $K(M)+K(H)$. The announced result follows from Brunn-Minkowski's inequality.
\end{proof}

\begin{proposition}
\label{prop:ComparisonPrinciple}
	If $u_*$ is a sub-solution of (\ref{eqdef:Pbn}, right), and $u^*$ a super-solution, then $u_* \leq u^*$.
\end{proposition}

\begin{proof}
The result does not immediately follow from \cite{Crandall:1992kn} because the operator 
$-D(\nabla^2 u)$
is only degenerate elliptic. The following operator is in contrast strictly elliptic when $\ve>0$: 
\begin{equation}
	F_\ve(u) := -D(\nabla^2 u)^\frac 1 d + \ve u.
\end{equation}
Let $M := \sup_{\overline \Omega} u_*-u^*$, which is finite by upper-semi-continuity. For contradiction, we assume $M>0$. Let also $r := \max \{\|x\|; x \in \overline \Omega\}$ and 
\begin{equation}
	u_\ve(x) := u_*(x) + \frac {\ve M} 2 (\|x\|^2-r^2). 
\end{equation}
One has $u_\ve \leq u_* \leq u^*$ on $\partial \omega$, and by Lemma \ref{lem:MinkowskiD} 
\begin{align}
	F_\ve(u^*)-F_\ve(u_\ve) \geq \ve M -\ve(u_* - u^*) \geq 0.
\end{align}
Applying the standard comparison principle \cite{Crandall:1992kn} to the strictly elliptic $F_\ve$ we obtain $u^* \geq u_\ve$, hence $u^*-u_* \geq  - \ve M r^2/2$.  Letting $\ve \to 0$ yields $u^* \geq u_*$ as announced.
\end{proof}

Weal solutions can be extracted from sequences of discrete solutions:
for each $n \geq 1$, extend $u_n$ by bilinear interpolation (or any other local interpolation procedure) on $\Omega_n := \{x \in \Omega; \, d(x, \partial \Omega) \geq n^{-1} \sqrt d\}$, and define $u_*, u^* : \overline \Omega \to \mR$ by 
\begin{align}
	u_* (x) &:= \limsup_{n \to \infty} u_n(x), & u^*(x) &:= \liminf_{n \to \infty} u_n(x).
\end{align}

\begin{lemma}
$u_*$ and $u^*$ are respectively a sub-solution and a super-solution of (\ref{eqdef:Pbn}, right). 
\end{lemma}

\begin{proof}
	Inspection of the proof of Corollary \ref{corol:Proper} yields the quantitive estimate
	\begin{equation}
		\min_{\partial \Omega} \sigma - \left(\frac 1 {n^d} \sum_{x \in X_n} \frac{\rho(x)}{\omega_d}\right)^\frac 1 d \diam(\Omega) \leq u_n \leq \max_{\partial \Omega} \sigma.
	\end{equation}
	Note that $n^{-d} \sum_{x \in X_n} \rho(x) \to \int_\Omega \rho$ as $n \to \infty$. (Also, the multiplicative term $h_*$ appearing in Corollary \ref{corol:Proper} equals $1$ with our choice of discrete domain $X_n = \Omega \cap n^{-1}\mZ^d$, since $]x,x+e/n] \cap (X_n \cup \partial \Omega)$ is non-empty for all $x \in X_n$, $e\in \mZ^d$.)
	
	The maps $u_n \in \mU_n$ are therefore bounded independently of $n$, hence $u_*$ and $u^*$ are well defined. From this point, the announced result follows by a standard argument: let $x \in \Omega$ and $\vp \in C^2(\Omega)$ be such that $u^*-\vp$ attains a strict global maximum at $x\in \Omega$. For each $n \geq 1$, let $x_n$ be the element of $X_n$ which maximizes $u_n-\vp$. By monotonicity $\cD_n \vp(x_n) \geq \cD_n u_n (x_n) = \rho(x_n)$. Observing that $x_n \to x$ as $n \to \infty$, we obtain $D(\nabla^2 \vp(x)) \geq \rho(x)$ in the limit. This establishes that $u^*$ is a super-solution, and likewise $u_*$ is a sub-solution.
\end{proof}

\begin{corollary}
\label{corol:ExistenceContinuous}
	One has $u_* = u^*$, and this map is the unique solution $u_\infty$ of (\ref{eqdef:Pbn}, right). As a result, $u_n$ converges pointwise to $u_\infty$ as $n \to \infty$.
\end{corollary}

\begin{proof}
	By construction $u_* \leq u^*$, but by the comparison principle $u_* \geq u^*$, see Proposition \ref{prop:ComparisonPrinciple}. 
	Therefore (\ref{eqdef:Pbn}, right) admits the solution $u_\infty := u_*=u^*$ which, again by the comparison principle, is its unique solution. Finally, for any $x \in \Omega$ on has as $n\to \infty$: $\limsup u_n(x) =u^*(x)=u_*(x)= \liminf u_n(x)$, hence $u_\infty(x) = \lim u_n(x)$.
\end{proof}


\section{Numerical experiments}
\label{sec:Num}

We implemented the three%
\footnote{%
The filtered method \cite{Froese:2013ez}, which accuracy may be competitive, was not implemented due to lack of time.
}
Monge Ampere discretizations described in the introduction: the Finite Differences (FD) scheme $\cD^\FD$, the Wide Stencil (WS) scheme $\cD^\WS_\cB$, and the proposed scheme $\cD_V$. All presented experiments are three dimensional. The latter two schemes require choosing a collection $\cB$ of triplets of orthogonal vectors, or a stencil $V$, see Table \ref{table:Stencils}. 
We recall that scheme FD is consistent but not monotone (or degenerate elliptic). Scheme WS is monotone and thus  benefits from the associated convergence guarantees, but suffers from significant consistency errors, see Figure \ref{fig:Relative error}.
The proposed scheme is simultaneously monotone and consistent, provided the PDE solution hessian condition number is bounded, and the scheme stencil $V$ is sufficiently large, see Proposition \ref{prop:Consistency} and Corollary \ref{corol:FiniteStencil}. A consistency error arises when these conditions are not satisfied, see Figure \ref{fig:Relative error}. 

We limit our attention to synthetic test cases, posed on the unit cube $\Omega := ]0,1[^3$. A known convex function $U: \overline \Omega \to \R$ is numerically recovered from its hessian determinant $\rho := \det(\nabla^2 U)$, and its boundary values $\sigma := U_{|\partial \Omega}$. 
Three test cases are considered.
\begin{itemize}
\item (Quadratic) $U(x) := \frac 1 2 \<x, M x\>$, where $M$ was chosen randomly, with $\sqrt {\|M\| \|M^{-1}\|}\approx 8.5$.
\item (Smoothed cone) $U(x) := \sqrt{\delta^2+ \|x-x_0\|^2}$, with $\delta := 0.1$ and $x_0:=(1/2,1/2,1/2)$. 
\item (Singular, \cite{Froese:2013ez})  $U(x) := -\sqrt{3-\|x\|^2}$.
\end{itemize}
A Damped Newton solver is applied to the discrete system \eqref{eq:DiscreteSys}, starting from the trivial seed $u(x) := \|x\|^2$. 

Quadratic test case. The chosen matrix $M$ does not belong to the consistency set of schemes $\cD_V$ and $\cD^\WS_\cB$, with the chosen $V$ and $\cB$; hence the numerical error reflects their consistency error, and is resolution independent. On the topic of consistency, it was determined using Proposition \ref{prop:Consistency} and some semi-definite programming that the consistency set of the proposed scheme $\cD_V$ contains all $M\in S_3^+$ such that $\Tr(M)/(\det M)^\frac 1 3$ is less than $7.8$ with the small stencil $V$, and less than $11.9$ with the large $V$, see Table \ref{table:Stencils}. The always consistent but non-monotone scheme $\cD^\FD$ finds the exact solution for a range of resolutions, up to machine precision, but switches to a completely erroneous solution at other resolutions.

Smoothed Cone test case. The test function $U$ is smooth, but the test is not as simple as it looks since (i) $\nabla U$ varies quickly close to the center $x_0$, but (ii) $\nabla^2 U$ is almost degenerate far from $x_0$. Point (i) favors small stencils, while point (ii) requires a good angular resolution. Scheme WS is most efficient with a small triplets collection $\cB$ at low resolutions, but with a medium or large one at high resolutions. Picking a stencil $V$ is easier with the proposed scheme: the larger one here always yields a smaller error, albeit at a higher computational cost. The non-monotone scheme FD passes this test as well.
We give some computation times for this test, which are typical of our experience. On a 2.7 Ghz Laptop using a single core, using a $50 \times 50 \times 50$ discretization grid $X$, computations took (in minutes): $4.7$ and $9.7$ with the proposed scheme, small and large $V$ respectively; $6.4$, $24$ and $65$ with scheme WS, small medium and large $\cB$ respectively; $6.5$ with scheme FD.

Singular test case. The function $U$ is not smooth at the point $(1,1,1)$, where its gradient is formally $(+\infty, +\infty, +\infty)$. The non-monotone scheme FD completely fails this test: numerical error does not decrease as resolution increases, and initializing the Newton solver with the exact solution $U$ does not even help, as observed in \cite{Froese:2013ez}. Scheme WS will recover the solution, if the bases collection $\cB$ is large enough, but numerical error decays quite slowly. The proposed scheme works flawlessly: it has the same convergence guarantees as scheme WS, but yields an $L^\infty$ error about $20$ times smaller on a $50^3$ grid. Curiously, numerical error is identical with the small and the large stencil $V$.

\begin{table}
\begin{tabular}{|c|cc|}
\hline
Proposed & Stencil vectors & $\#(V)$\\
\hline 
Small & $(1,0,0), (1,1,0), (1,1,1)$ & 13 \\
Large & Same and $(2,1,0), (2,1,1)$ & 37\\
\hline
\end{tabular}
\hspace{0mm}
\begin{tabular}{|c|cc|}
\hline
Scheme WS & Orthogonal triplets & $\#(\cB)$\\ 
\hline 
Small & All within $\{-1,0,1\}^3$ & 6 \\
Medium & All within $\{-2, \cdots, 2\}^3$ & 43 \\
Large & All within $\{-3, \cdots, 3\}^3$ & 82\\
\hline
\end{tabular} 

\caption{\emph{Left:} Stencil $V$ used with the proposed scheme $\cD_V$. In addition to those indicated $(\alpha_1, \alpha_2, \alpha_3)$, the stencil contains all vectors $(\ve_1 \alpha_{\sigma(1)}, \ve_2 \alpha_{\sigma(2)}, \ve_3 \alpha_{\sigma(3)})$ obtained by permuting and changing the sign of their coordinates. Opposite vectors are identified  when evaluating $\#(V)$.
\newline
\emph{Right:} Collection $\cB$ of orthogonal triplets used with the Wide-Stencil scheme $\cD^\WS_\cB$. Triplets are counted up to reordering their components, and changing their sign: $(e_0, e_1, e_2) \sim (\ve_1 e_{\sigma(1)}, \ve_2 e_{\sigma(2)}, \ve_3 e_{\sigma(3)})$. Vectors with non-coprime coordinates, such as $(2,2,0)$, are rejected. 
}
\label{table:Stencils}
\end{table}

\begin{figure}
\centering
\begin{tabular}{cc|ccc}
\text{Proposed, Small} & \text{Proposed, Large} & \text{WS, Small} &  \text{WS, Medium} &  \text{WS, Large} \\ 
\includegraphics[width=2.9cm]{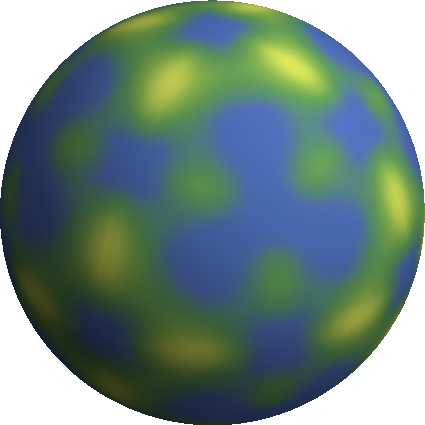} &
\includegraphics[width=2.9cm]{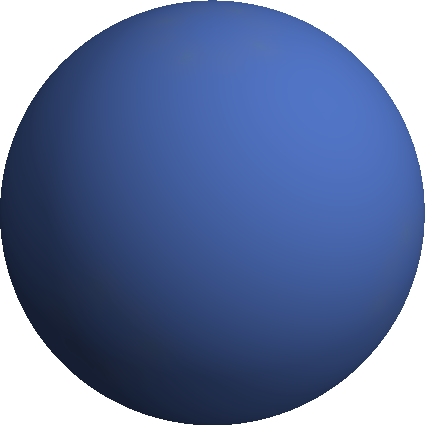} &
\includegraphics[width=2.9cm]{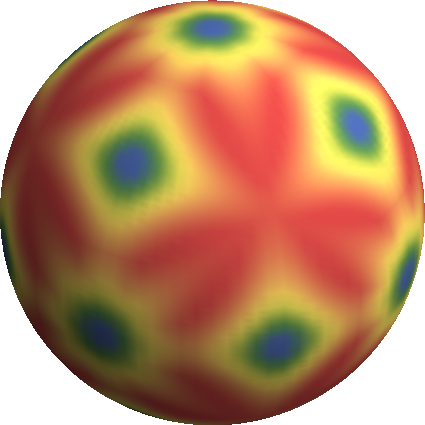} &
\includegraphics[width=2.9cm]{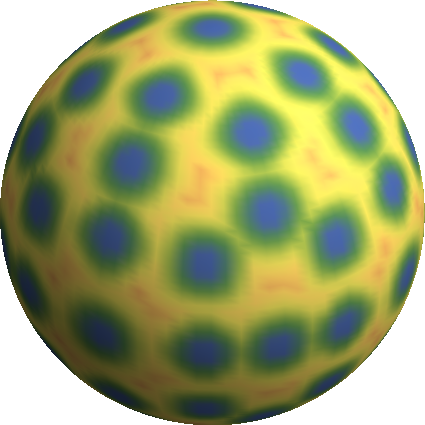} &
\includegraphics[width=2.9cm]{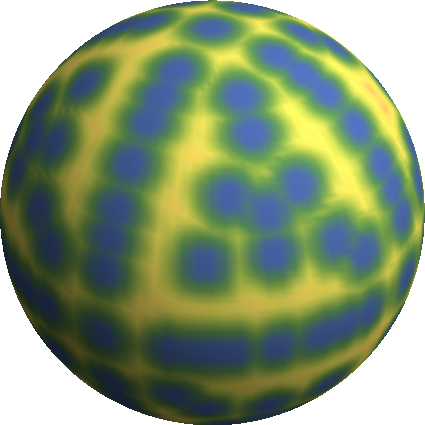} \\
\includegraphics[width=2.9cm]{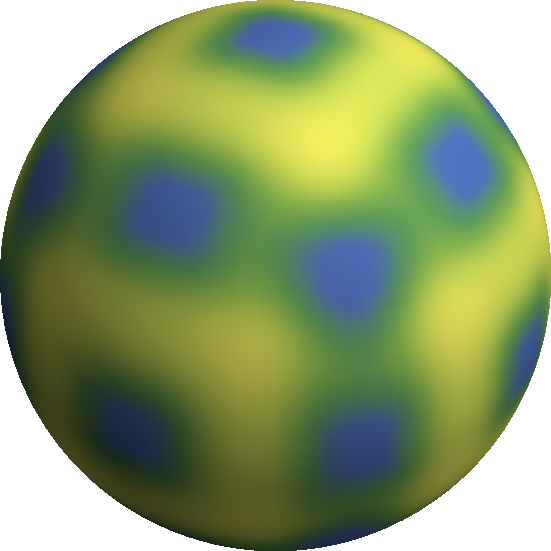} &
\includegraphics[width=2.9cm]{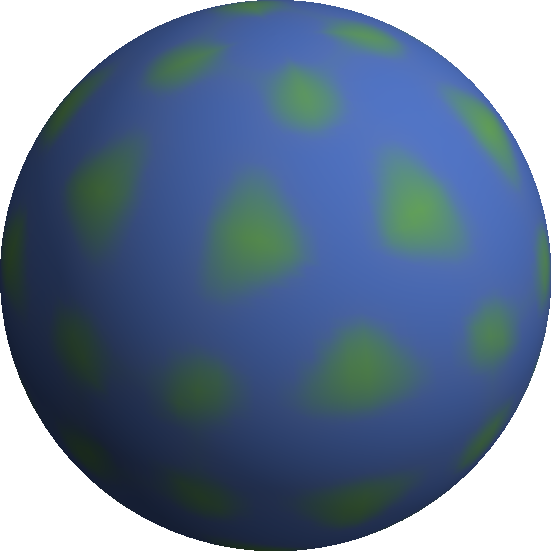} &
\includegraphics[width=2.9cm]{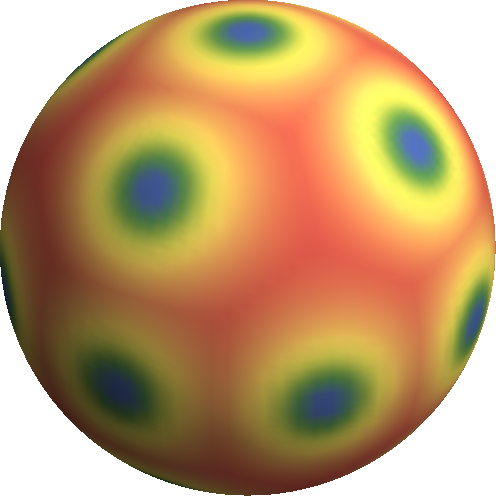} &
\includegraphics[width=2.9cm]{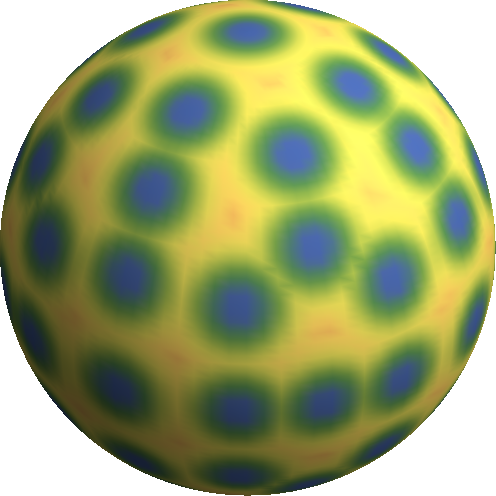} &
\includegraphics[width=2.9cm]{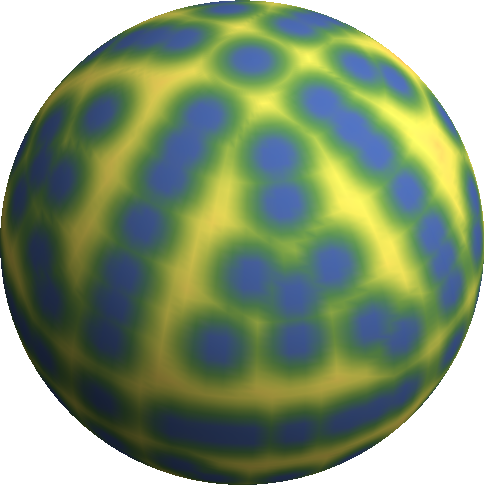} \\
\includegraphics[width=2.9cm]{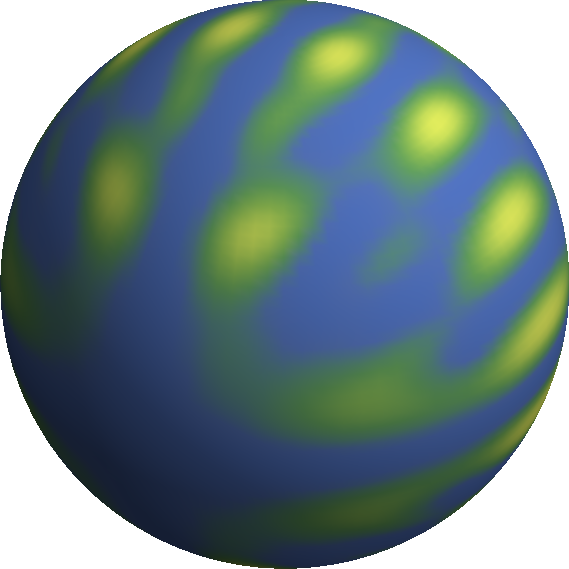} &
\includegraphics[width=2.9cm]{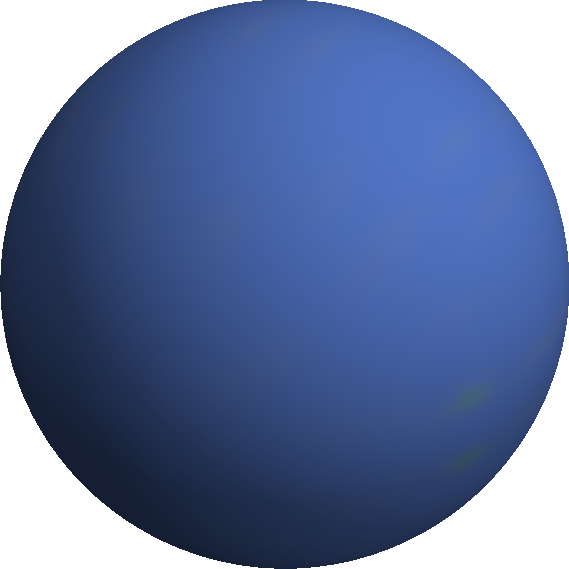} &
\includegraphics[width=2.9cm]{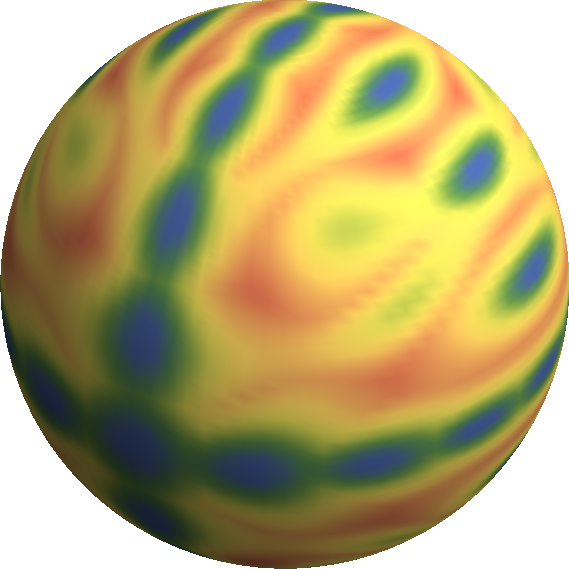} &
\includegraphics[width=2.9cm]{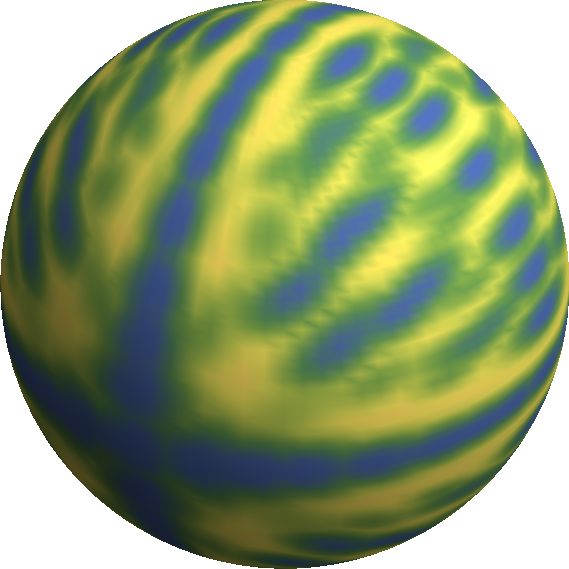} &
\includegraphics[width=2.9cm]{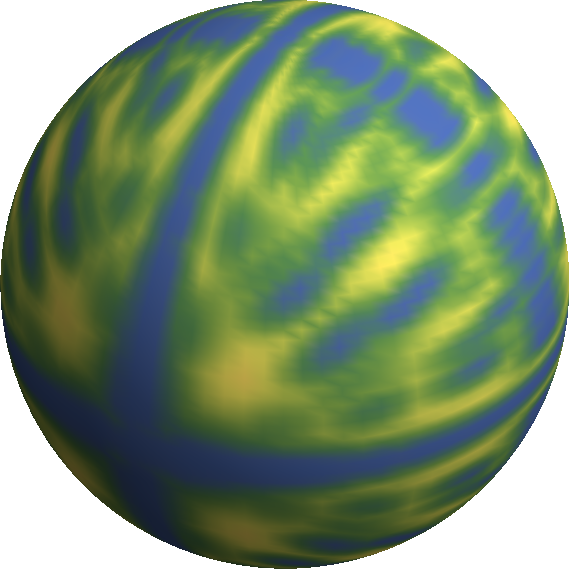} 
\end{tabular}
\caption{
Relative consistency error $(\cD u_M - \det(M) ) / \cD u_M$ ($\text{blue}=0$, $\text{red}=1$) for a quadratic \eqref{eqdef:MNorm} map $u_M$ associated to a matrix $M = M(v)$, with $v$ in the unit sphere (hence the spherical plot). 
First line: $M(v)$ has eigenvalues $6^2, 1,1$, the former with eigenvector $v$. Second line likewise with eigenvalues $6^{-2},1,1$. Third line $M(v) = R(v) D R(v)$ where $R(v)$ is the rotation of axis $v$ and angle $\pi$, and $D$ is diagonal of entries $6,1,1/6$.}
\label{fig:Relative error}
\end{figure}

\begin{figure}
\centering
\includegraphics[width=3.5cm]{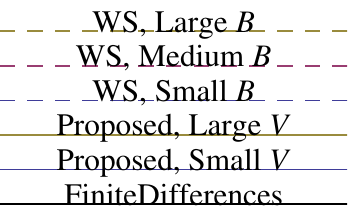}
\includegraphics[width=4cm]{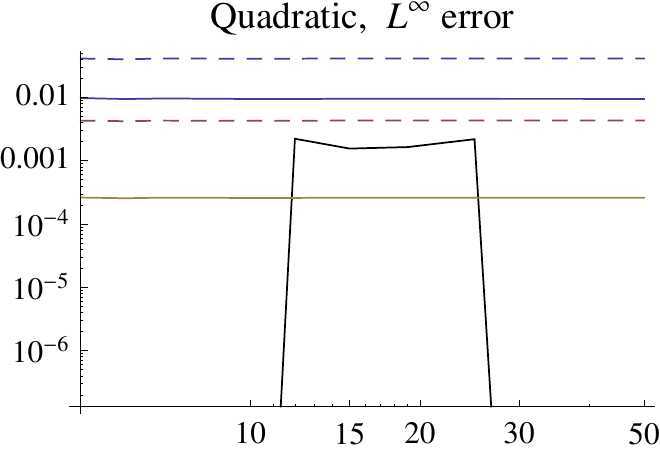}
\includegraphics[width=4cm]{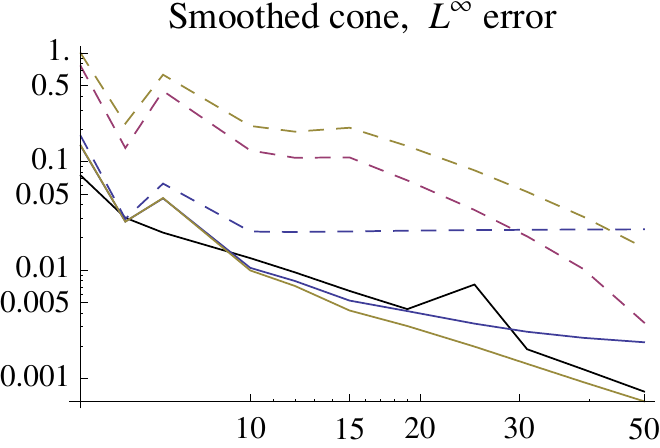}
\includegraphics[width=4cm]{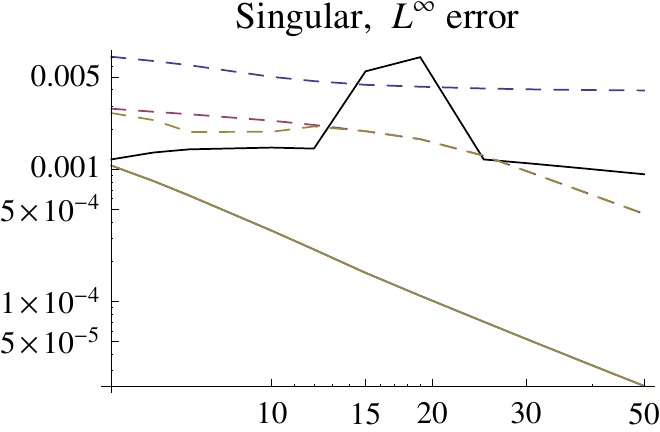}
 
\caption{
Log-Log plot of the $L^\infty$ error, as a function of resolution $n$, for the three test cases. 
Discretization set $X\subset \Omega$ has $n^3$ points.
Scheme stencils on Figure \ref{table:Stencils}.
}
\label{fig:ErrorPlots}
\end{figure}

\bibliographystyle{alpha}
\bibliography{\pathBib/AllPapers}

\end{document}